\documentclass{article}

\usepackage[margin=1.3in]{geometry}

\usepackage{amsmath,amssymb,amsthm,mathrsfs,esint,color,times,textcomp,sectsty,verbatim,graphicx,enumerate,authblk,amsxtra}
\usepackage{xcolor}
\usepackage[colorlinks=true]{hyperref}
\hypersetup{urlcolor=blue, citecolor=red, linkcolor=blue}
\usepackage[toc,page]{appendix}
\numberwithin{equation}{section}
\theoremstyle{plain}
\newtheorem{theorem}{Theorem}[section]

\newtheorem{lemma}[theorem]{Lemma}

\theoremstyle{definition}
\newtheorem{definition}[theorem]{Definition}

\newtheorem{remark}[theorem]{Remark}

\newcommand{\mr}{\mathbb{R}}
\newcommand{\ud}{\mathrm{d}}
\allowdisplaybreaks

\begin{document}
	\title{\Large \bf Asymptotic behavior of conformal metrics with null  Q-curvature}
	\author{Mingxiang Li\thanks{M. Li, Nanjing University,  Email: limx@smail.nju.edu.cn. }}
	\date{}
	\maketitle
	\begin{abstract}
	We  describe the asymptotic behavior of conformal metrics related to the GJMS operator in the null case, as the prescribed Q-curvature $f_0(x) + \lambda$ gradually changes. We show that if  one of the maximum points of $f_0$ is flat up to order $n-1$, the normalized conformal metrics in the lowest energy level will form exactly one spherical bubble as $\lambda$ approaches zero using higher order Bol's inequality. This generalizes the result of Struwe (JEMS, 2020) in the two-dimensional case to higher dimensions and helps rule out the slow bubble case discussed by Ngô and Zhang (arXiv:1903.12054) to some degree.
	\end{abstract}

\medskip

\section{Introduction}

	Given a compact Riemann  surface $(M^2,g)$, the prescribed Gaussian curvature problem in conformal geometry attracts a lot of interest. It is equivalent  to solve the following conformal equation
	\begin{equation}\label{equ: prescribed Gaussian curvature}
		-\Delta_gu+K_g=fe^{2u}
	\end{equation}
	 where $\Delta_g$ is the Laplace-Beltrami operator, $K_g$ is Gaussian curvature of $(M^2,g)$ and $f$ is the prescribed smooth function  defined on $M^2$. The special case where the function $f$ is a constant has been extensively studied  due to the Uniformization theorem. 
	 In the more general case, the problem becomes more complex and has been the subject of ongoing research in the field. The seminal work of Kazdan and Warner \cite{KW} has played a crucial role in the development of this problem  and we highly recommend referring to it for further details.  The equation \eqref{equ: prescribed Gaussian curvature} is commonly known as the Nirenberg problem when $(M^2,g)$ is a standard sphere. In their seminal work \cite{KW}, Kazdan and Warner provided an obstruction, known as Kazdan-Warner's identity, to the existence of solutions to \eqref{equ: prescribed Gaussian curvature}, which significantly increases the difficulty of the Nirenberg problem. Many important contributions have been made towards the solution of this problem, and we refer the interested reader to \cite{Chang-Yang2}, \cite{CYJDG} and many others. In the case of a compact four-dimensional manifold $(M^4,g)$,  Paneitz \cite{Paneitz} introduced a conformal operator defined by
$$P_g^4=\Delta_g^2-\mathrm{div}_g\left((\frac{2}{3}R_gg-2\mathrm{Ric}_g)d\right)$$
	where $R_g$ and $\mathrm{Ric}_g$ denote the scalar curvature and Ricci curvature tensor  of $(M^4,g)$,  respectively. Later, Branson \cite{Bra} introduced the Q-curvature defined  by
	$$Q_g=-\frac{1}{12}(\Delta_gR_g-R_g^2+3|\mathrm{Ric}_g|^2).$$
	 Interestingly, similar to the equation \eqref{equ: prescribed Gaussian curvature}, the prescribed Q-curvature problem is equivalent to solving
	 \begin{equation}\label{equ:prescribed Q-curvature}
	 	P_g^4u+Q^4_g=fe^{4u}
	 \end{equation}
	for the conformal metric $\tilde g=e^{2u}g$ where $Q_g^4=2Q_g$. 
	 For higher order cases, Graham, Jenne, Mason and Sparling \cite{GJMS} introduced an operator known as GJMS operator  $P_g^{n}$ for $(M^n,g)$ satisfying
	 \begin{equation}\label{equ: GJMS operator}
	 	P_g^nu+Q^n_g=Q^n_{\tilde g}e^{nu}
	 \end{equation}
	 where $\tilde g=e^{2u}g$ where $n\geq 4$ is an even integer.
	We refer the interested reader to \cite{CY95}, \cite{DM},  \cite{Ma}, \cite{Nid}  and the references therein for more details.
	
In the current paper, our focus is on the null case, i.e., $Q_g^n\equiv 0$, and we aim to describe the asymptotic behavior of a family of conformal metrics. To facilitate better understanding, we begin with the two-dimensional case.	When the background Gaussian curvature vanishes, Kazdan and Warner \cite{KW} provided a sufficient and necessary condition for the existence of solutions to \eqref{equ: prescribed Gaussian curvature}: $f$ changes sign and $\int_Mf\ud\mu_g<0$ or $f\equiv 0$. It is natural to ask what will happen to the conformal metrics if this condition gradually  violated?  Let us expain it more precisely. Consider  a non-constant smooth function $f_0\leq 0$ on $M$ that vanishes at some points. For $\lambda\in (0,-\frac{1}{vol(M)}\int_Mf_0\ud\mu_g)$, it is easy to check that $f_\lambda:=f_0+\lambda$ satisfies Kazdan-Warner condition, and thus the solutions to \eqref{equ: prescribed Gaussian curvature} exist. 
Galimberti's study in \cite{Ga} focused on the asymptotic behavior of the minimizer of the Dirichlet energy where he found  that as $\lambda$ approaches  zero, $\beta(\lambda)$ tends to infinity monotonically and  the  metrics blow up at finite points. Additionally, he provided several possible characterizations of the bubbles (Theorem 1.2 in \cite{Ga}). Moreover, Galimberti presented a complete description of the case where $\lambda \to-\frac{1}{\text{vol}(M)}\int_Mf_0\ud\mu_g$ and he demonstrated that $\beta(\lambda)$ tends to zero while the metrics  collapse. Subsequently, Struwe \cite{Struwe20} provided a more detailed characterization of the former case by ruling out the possibility of the "slow bubble" scenario:
 \begin{equation}\label{equ: slow bubble}
 	-\Delta_{\mr^2} u=(1+A(x,x))e^{2u}
 \end{equation}
where $A$ is a negative $2\times 2$ matrix  with   $(1+|A(x,x)|)e^{2u}\in L^1(\mr^2)$ which is 
 allowed in Galimberti's theorem (See Theorem 1.1 in \cite{Ga}). In fact, Struwe established a Liouville-type result (refer to Theorem 1.3 in \cite{Struwe20}) that proved the nonexistence of a solution to \eqref{equ: slow bubble}.  He then demonstrated that when $f_0$ is non-degenerate at each maximum point, the metrics can only form at most two spherical bubbles. This result was further improved by the author in \cite{LX23}, in collaboration with Xu, where we showed that the metrics will shape exactly one spherical bubble even in more general cases based on a global version of Alexandrov-Bol's inequality.  Meanwhile, for the general Liouville-type theorem related  to  the equation \eqref{equ: slow bubble} was established in \cite{Li23}. For a similar phenomenon in compact surfaces of higher genus, there have been several studies in the literature, such as \cite{BGS}, \cite{DL}, and \cite{dR}, among others. 
 
In the case of four-dimensional manifolds, Ngô and Zhang \cite{NZ19} also studied a similar phenomenon related to Q-curvature based on an existence result for \eqref{equ:prescribed Q-curvature} (see Theorem 1.1 in \cite{GX}). They showed that the "slow bubble" may occur in the fourth dimension when the maximum points of $f_0$ are non-degenerate. On the other hand, Hyder and Martinazzi \cite{HM} showed that the following equation
 \begin{equation}\label{equ: slow bubble in fourth case}
 	\Delta^2u(x)=(1-|x|^2)e^{4u(x)}
 \end{equation}
with $(1+|x|^2)e^{4u}\in L^1(\mr^4)$ does have solutions.  There are various indications that the "slow bubble" does exist. However, it is interesting that  we will show that the  metrics will  shape exactly one spherecial bubble   in the higher order cases   if there exists a maximum point of  $f_0$ is flat up to $n-1$ order. This means "slow bubble" doesn't exist in this case.  It is worth noting that  we allow some maximum points to be non-degenerated.
 We will utilize the sharp upper bound of energy in Lemma \ref{lem: beta(lambda) upper boud} and higher order Bol's inequality developed in \cite{LW} to rule out the "slow bubble".   
 For convenience, we give  some definitions before  stating our main theorem.
 \begin{definition}\label{def: 2l max point}
 	We say a maximum point $P_0$ of $f\in C^\infty(M)$ is of $l-$type if in the small  neighborhood of $P_0$, considering  the exponential map $\exp_{P_0}: B_{\delta}(0)\subset\mr^n\rightarrow M$, $f\circ\exp_{P_0}(z)$ has the expansion
 	$$f\circ\exp_{P_0}(z)=f(P_0)+p_{l}(z)+O(|z|^{l+1})$$
 	where $p_{l}(z)\leq 0$ is a  homogeneous polynomial of degree $l$ and $p_{l}(z)\not=0$ if $z\not=0$.
 \end{definition}
 \begin{definition}
 	We say $f\in C^\infty(M)$ only has isolated maximum points if each maximum point is $l-$type for some integer $l\geq 2$ and there are only finitely many such maximum points.
 \end{definition}
\begin{definition}
	We say a sequence $\{u_k\}\subset H^{\frac{n}{2}}(M,g)$ forms a spherecial bubble   at $P_0$ if there exists a subsequence as well as  $r_k\to 0$ and $x_k=\exp_{P_0}(z_k)\to P_0$ such that
	$$u_k\circ\exp_P(r_k z+z_k)+\log r_k\rightarrow \log\left(\frac{1}{1+|z|^2}\right),$$
	up to a translation, a scaling or adding some constants,	in $H^{n/2}_{loc}(\mathbb{R}^n).$
\end{definition}

For the sake of convenience, we use the notation $B_R(p)$ to refer to an Euclidean  ball in $\mr^n$ centered at $p\in\mr^n$ with a radius of $R$. Also, $|B_R(p)|$ refers to the volume of $B_R(p)$ in terms of the Euclidean metric.  For a function $\varphi(x)$, the positive part of $\varphi(x)$ is denoted as $\varphi(x)^+$ 
and the negative part of $\varphi(x)$ is denoted as $\varphi(x)^-$.  Set $\fint_{E}\varphi(x)\ud x=\frac{1}{|E|}\int_{E}\varphi(x)\ud x$ for any measurable set $E$.   We denote by $C$ a constant which may be different from line to line. $[s]$ denotes the largest integer not greater than $s$.

Suppose that  $(M,g)$ is a compact manifold of even dimension $n\geq 4$ and  GJMS operator (denoted as $P_g $ for simplicity) is positve with kernel consisting of constants and Q-curvature vanishes.
 Consider a class of smooth functions on $M$ defined as
$$\mathcal{F}:=\{f\in C^\infty(M)| f\leq 0\;\mathrm{and}\;f\not\equiv0, \exists\; x_0\in M, f(x_0)=0      \}$$  	
Given	 $f_0\in\mathcal{F}$,  let
	$f_\lambda(x)=f_0+\lambda$
	where $\lambda\in(0,-\frac{1}{vol(M)}\int_Mf_0\ud\mu_g)$. We define the energy functional 
	$$E(u):=\int_M uP_gu\ud \mu_g,$$ 
	under  the constraint 
	$$\mathcal{M}(\lambda) :=\{u\in H^{\frac{n}{2}}(M,g)|\int_Mf_\lambda e^{nu}\ud\mu_g=0\}, 
	\quad \mathcal{M}^*(\lambda) :=\{u\in \mathcal{M}(\lambda)|\int_M e^{nu}\ud\mu_g=1\}$$
	where $H^{\frac{n}{2}}(M,g)$ denotes the Hilbert  space.
	Based on Theorem \ref{thm: minimizer} below and the property $E(u+c)=E(u)$ for any constant $c$, a minimizer 
	$u_\lambda$ of $E(u)$ in $\mathcal{M}^*(\lambda)$  is achiveved satisfying
	\begin{equation}\label{equ:u_lambda}
		P_gu_\lambda=\alpha_\lambda f_\lambda e^{nu_\lambda}
	\end{equation}
	where $\alpha_\lambda>0$ and set  
	$$\beta(\lambda):=E(u_\lambda).$$
	Now, we state our theorem in current paper.
	\begin{theorem}\label{main theorem}Consider a compact manifold  $(M,g)$ of even dimension $n\geq 4$. Suppose that   the GJMS operator of $(M,g)$  is positive  with kernel consisting of constants and  Q-curvature vanishes.
		Given  a smooth  function $f_0\in\mathcal{F}$ which only has  isolated maximum points including  an $l$-type maximum point with $l\geq n$, then  there exists  a sequence $\lambda_k\downarrow 0$ such that $\{u_{\lambda_k}\}$ shapes only one spherecial bubble at some maximum point $p_0$ of $f_0$. Moreover,
		$u_{\lambda_k}(x)\to -\infty$
		on any compact subset of $M\backslash \{p_0\}$ and 
		$$\lim_{k\to\infty}\lambda_k\alpha_{\lambda_k}=\frac{n}{2}\lim_{k\to\infty}\frac{\beta(\lambda_k)}{\log(1/\lambda_k)}=(n-1)!|\mathbb{S}^n|$$
		where $|\mathbb{S}^n|$ denotes the volume of unit sphere in $\mathbb{R}^{n+1}$.
	\end{theorem}

During doing  blow-up analysis, we  need to study the following conformally invariant  equation: 
\begin{equation}\label{equ:-Delta ^n/2u=Ke^nu}
	(-\Delta)^{\frac{n}{2}}u(x)=Q(x)e^{nu(x)}\quad \mathrm{on}\;\;\mr^n
\end{equation}
where  $n\geq 2$ is an even integer and $Q(x)$ is a smooth function. Supposing that $Qe^{nu}\in L^1(\mr^n)$,
we say $u(x)$ is a normal solution to \eqref{equ:-Delta ^n/2u=Ke^nu} if $u(x)$ satisfies the integral equation 
\begin{equation}\label{normal solution}
	u(x)=\frac{2}{(n-1)!|\mathbb{S}^n|}\int_{\mr^n}\log\frac{|y|}{|x-y|}Q(y)e^{nu(y)}\ud y+C
\end{equation}
for some constant $C$. More details about normal solutions can be found in Section 2 of \cite{Li 23 Q-curvature}.  Here, we assume that $u(x)$ and $Q(x)$ are both smooth functions on $\mathbb{R}^n$. Although similar results can also be obtained under weaker regularity assumptions,  for the sake of brevity, we will focus on the smooth case throughout  this paper.

In \cite{LX23}, the Alexandrov-Bol's inequality  was used to rule out the "slow bubble" in the two-dimensional case.  For brevity, 
Alexandrov-Bol's inequality on $\mr^2$ shows that if the Guassian curvature has an upper bound, the volume  has a sharp  lower bound. More details can be found in \cite{Ba}, \cite{LW} and the references therein.  In the context of higher-order cases, a natural inquiry arises concerning the existence of analogous Bol's inequalities. In \cite{LW}, the author and J. Wei provide a partial response to this question. Leveraging the findings established in \cite{LW}, we can effectively eliminate the presence of slow bubbles.

Now, we will briefly outline the structure of this paper. In Section \ref{section 2}, we provide another proof of the existence result for null Q-curvature related to GJMS operator. Additionally, we demonstrate the monotonicity of the minimizing energy and its growth control with respect to $\lambda$. In Section \ref{section 3}, we carry out a  blow-up analysis for a sequence of normalized minimizers of energy.  Finally,  in Section \ref{section 4}, we present a proof of our main Theorem \ref{main theorem} with help of higher order Bol's inequality.

\section{Existence of minimizer and monotonicity of energy}\label{section 2}
	Firstly,  given a sign-changing smooth function $f(x)$ on $(M,g)$,  we  give  a proof of the existence of the minimizer of energy $E(u)$ under the constraint 
	$$\mathcal{M}:=\{u\in H^{\frac{n}{2}}(M)|\int_M fe^{nu}\ud\mu_g=0\}.$$
	Actually, it has been proved by Ge and Xu in \cite{GX}. Besides, Ng\^o and Zhang in \cite{NZ} gave another proof in their appendix.  However, our proof has its own interest since the continuity trick during the proof of  monotonicity of energy will play a very important role.   Meanwhile,	compared with Corollary 2.4 in \cite{BFR}, we remove the restriction 
	$$\int_M|\Delta_gu|^2\ud\mu_g< \sigma,$$
	for some positive constant $\sigma$ depending on $(M,g)$.

\begin{theorem}\label{thm: minimizer}
Given  a smooth compact manifold $(M^n,g)$ with unit volume  where $n\geq 4$ is an even integer.  Assume that $P_g$ is positive  GJMS operator  with kernel consisting of constant functions. For any smooth and sign-changing function $f(x)$,  the  minimizers  of  $E(u)$ with $u\in \mathcal{M}$  exist. If we choose a minimizer  denoted as $u_0$,  up to a constant, $u_0$ satisfies  
\begin{equation}\label{equ:minimizer u_0}
	P_gu_0=sgn(-\bar f)fe^{nu_0}
\end{equation}
	where $\bar f=\int_Mf\ud\mu_g$ and $sgn$ is the sign function.
\end{theorem}
\begin{proof}
The case $\bar f=0$  is trivial since all  constant functions are the minimizers lying in $\mathcal{M}$ as well as the assumption on the kernel of $P_g$.

 As for $\bar f\not=0$, we just need consider the case $\bar f<0$ since we can do the same following programe by considering $-f$.
	Suppose $f(x_0)=\max_M f(x)>0$ and then
	there exists $\delta>0$ such that $f(x)\geq \frac{1}{2}f(x_0)$ for any $x\in B_{\delta}(x_0)$.
	We could  choose a non-negative cut-off function $\eta(x)$ such that $\eta(x)$ equals to $c$ in $B_{\delta/2}(x_0)$ to be chosen later and vanishes outside $B_\delta(x_0)$.
	 Then direct computation yields that 
	 \begin{align*}
	 \int_M fe^{n\eta}\ud\mu_g=&\int_Mf(e^{n\eta}-1)\ud\mu_g+\int_M f\ud\mu_g\\
	 =&\int_{B_{\delta}(x_0)}f(e^{n\eta}-1)\ud\mu_g+\int_Mf\ud\mu_g\\
	 \geq&\int_{B_{\delta/2}(x_0)}\frac{f(x_0)}{2}(e^{nc}-1)\ud\mu_g+\int_M f\ud\mu_g>0
	 \end{align*}
	if we choose sufficiently large $c$.
	Set
	$$\varphi(t)=\int_Mfe^{tn\eta}\ud\mu_g$$
	where $t\in [0,1]$ and then immediately $\varphi(1)>0$ and $\varphi(0)<0$. It is not hard to check that  $\varphi(t)$ is continuous respect to $t$. Then there exists $t_0\in(0,1)$ such that $\varphi(t_0)=0$ which concludes that $t_0\eta\in\mathcal{M}.$  Thus we show that  $\mathcal{M}$ is not empty. 
	
	Due to $P_g$  positive, $E(u)$ is bounded from below by zero in $\mathcal{M}$. Noticing  that $E(u+c)=E(u)$,  we can normalize $\mathcal{M}$ as
	$$\mathcal{M}_f=\{u\in \mathcal{M}|\int_Me^{nu}\ud\mu_g=1\}.$$
	Thus  we have 
	\begin{equation}\label{inf E(u)}
		\mathcal{A}:=\inf_{u\in\mathcal{M}}E(u)=\inf_{u\in\mathcal{M}_f}E(u).
	\end{equation} 
	Suppose  a sequence $\{u_i\}\subset \mathcal{M}_f$ with $E(u_i)\leq C_0$ minimizing $E(u)$ where $C_0$ is a constant independent of $i$.
	Jensen's inequality implies that
	$$\bar u_i:=\int_Mu_i\ud\mu_g\leq\frac{1}{n} \log\int_Me^{nu_i}\ud\mu_g=0.$$
	Adams-Fontana inequality(See \cite{CY95},\cite{Nid}) yields that
	$$-n\bar u_i=\log\int_Me^{n(u_i-\bar u_i)}\ud\mu_g\leq C\int_Mu_iP_gu_i\ud\mu_g+C\leq C.$$
	Combing above two estimates, one has
	$$|\bar u_i|\leq C.$$
	Based on the  positivity assumption on $P_g$ as well as  interpolation inequality, we have
	$$\|u_i\|_{H^{n/2}(M)}\leq CE(u_i)+C|\bar u_i|\leq C.$$
	Then there exists a subsequnce of $\{u_i\}$ weakly converges to $u_\infty$ in $H^{\frac{n}{2}}(M)$ and strongly in $L^p(M) $ for any $1\leq p<+\infty$. Then Adams-Fontana's  inequality yields that for any $q>1$
	$$\int_Me^{qu_\infty}\ud\mu_g\leq C\exp(C(q)\|u_\infty\|_{H^{\frac{n}{2}}(M)}).$$
	 Following the argument in  Section 3 in \cite{KW} and the fact $|e^t-1|\leq |t|(e^{t}+e^{-t})$, there holds
	\begin{align*}
		&\int_M(e^{nu_i}-e^{nu_\infty})^2\ud\mu_g\\
		=& \int_Me^{2nu_\infty}(e^{n(u_i-u_\infty)}-1)^2\ud\mu_g\\
		\leq &\int_Me^{2nu_\infty}\left(e^{n(u_i-u_\infty)}+e^{-n(u_i-u_\infty)}\right)^2n^2|u_i-u_\infty|^2\ud\mu_g\\
		\leq &n^2(\int_Me^{8nu_\infty}\ud\mu_g)^{1/4}(\int_M\left(e^{n(u_i-u_\infty)}+e^{-n(u_i-u_\infty)}\right)^8\ud\mu_g)^{1/4}\|u_i-u_\infty\|_{L^4(M)}^2\\
		\leq & C\|u_i-u_\infty\|_{L^4(M)}^2\to 0
	\end{align*}
which concludes  that $u_\infty\in \mathcal{M}_f$. 
	Using the theory of Lagrange multiplier,  $u_\infty$ weakly solves the following equation
	\begin{equation}\label{equ:u-infty}
		P_gu_\infty=\lambda_1fe^{nu_\infty}+\lambda_2e^{nu_\infty}.
	\end{equation}
	Integrating  \eqref{equ:u-infty} over $M$, we  get $\lambda_2=0$. Obviously, $\lambda_1\not=0$ otherwise $u_\infty$ is a constant which contradicts to $\bar f<0$. Moreover, there holds $\mathcal{A}=E(u_\infty)>0$.
	
	We claim  $\lambda_1>0$ and argue by contradiction.  Supposing $\lambda_1<0$,  consider 
	$$\psi(t)=\int_Mfe^{ntu_\infty}$$
	where $t\in [0,1]$. 
	Taking the  same strategy as before and using the fact $|e^t-1-t|\leq t^2(e^t+e^{-t})$, for fixed  $t_0$ and $|\epsilon|<1$, we have 
	\begin{align*}
		&|\psi(t_0+\epsilon)-\psi(t_0)-\epsilon n\int_Mfe^{n t_0u_\infty}u_\infty\ud\mu_g|\\
		=&|\int_Mfe^{nt_0u_\infty}(e^{n\epsilon u_\infty}-1-n\epsilon u_\infty)\ud\mu_g|\\
		\leq &\int_M|f|e^{nt_0 u_\infty}\epsilon^2n^2u_\infty^2(e^{n\epsilon u_\infty}+e^{-n\epsilon u_\infty})^2\ud\mu_g\\
		\leq &C(f, n, t_0,  \|u_\infty\|_{H^{\frac{n}{2}}(M)})\epsilon^2
	\end{align*}
	Hence, $\psi(t)\in C^1$.
	Meanwhile, $\psi(1)=0$ due to $u_\infty\in \mathcal{M}_f$. Direct computation yields that
	$$\psi'(1)=n\int_Mu_\infty fe^{nu_\infty}\ud\mu_g=\lambda_1\int_Mu_\infty P_gu_\infty \ud\mu_g<0.$$
	Then there exists $t_1\in (0,1)$ such that $\psi(t_1)>0$. Recall $\psi(0)=\int_Mf\ud\mu_g<0$ and immediately there exists $t_2\in (0,t_1)$ such that $\psi(t_2)=0$. Thus $t_2u_\infty\in\mathcal{M}$. Although $t_2u_\infty$ doesn't belong to $\mathcal{M}_f$, the property \eqref{inf E(u)} shows that
	$$\mathcal{A}\leq E(t_2 u_\infty)=t_2^2\mathcal{A}<\mathcal{A}$$
	which is impossible. Hence $\lambda_1>0$.
	Then
	$\tilde u_\infty=u_\infty+\frac{1}{n}\log\lambda_1$
	satisfies
	$$P_g\tilde u_\infty=sgn(-\bar f)fe^{n\tilde u_\infty}.$$
\end{proof}

Now, we are going to use the  above continuity trick to show the monotonicity of energy respect to $\lambda$.
\begin{lemma}\label{trick lemma}
	$\beta(\lambda)$ is strictly decreasing in $(0,-\bar f_0)$.
\end{lemma}
\begin{proof}
	Given $0<\lambda_1<\lambda_2<-\bar f_0$,  by Theorem \ref{thm: minimizer}, we know that $\beta_{\lambda}$ is achieved by some $u_\lambda\in \mathcal{M}^*(\lambda)$ for any $\lambda\in [\lambda_1,\lambda_2]$. Set $$\xi(t)=\int_Mf_{\lambda_2}e^{n(1-t)u_{\lambda_1}}\ud\mu_g.$$ Notice that
	$$\xi(0)=\int_M(f_{\lambda_1}+\lambda_2-\lambda_1)e^{nu_{\lambda_1}}\ud\mu_g=\lambda_2-\lambda_1>0,$$
	and $\xi(1)<0$. Making use of the continuity of $\xi(t)$, there exists $t_1\in (0,1)$ such that $\xi(t_1)=0$. Hence
	$$\beta(\lambda_2)\leq\int_M(1-t_1) u_{\lambda_1}P_g(1-t_1)u_{\lambda_1}\ud\mu_g=(1-t_1)^2\beta(\lambda_1)<\beta(\lambda_1).$$
\end{proof}

Applying Lebesgue's Theorem, we observe that $\beta(\lambda)$ is almost everywhere differentiable. 
\begin{lemma}\label{Lemma beta_lambda derivative}
	$\beta(\lambda)$ is Lipschitz continuous on compact subset of $(0,-\bar f_0)$ and 
	$\beta'(\lambda)=-\frac{2\alpha_\lambda}{n}$ a.e.
\end{lemma}
\begin{proof}
	Multiply $f_\lambda$ to both sides of  the equation \eqref{equ:u_lambda} and integrate it over $M$ to get
	$$\alpha_\lambda=\frac{\int_M u_\lambda P_g f_\lambda\ud\mu_g}{\int_M f_\lambda^2e^{nu_\lambda}\ud\mu_g}=\frac{\int_M(u_\lambda-\bar u_\lambda)P_gf_0\ud\mu_g}{\int_M f_\lambda^2e^{nu_\lambda}\ud\mu_{g}}.$$
	Using H\"older's inequality, one has 
	$$0<\|f_\lambda\|^2_{L^1}\leq\left(\int_Mf_\lambda^2e^{nu_\lambda}\ud\mu_{g_b}\right)\int_Me^{-nu_\lambda}\ud\mu_{g_b}.$$
	Then
	$$\alpha_\lambda\leq C\frac{\|u_\lambda-\bar u_\lambda\|_{L^2}\|f_0\|_{H^n}}{\|f_\lambda\|^2_{L^1}}\int_Me^{-nu_\lambda}\ud\mu_g.$$
	With help of Adams-Fontana's inequality and $u_\lambda\in \mathcal{M}^*(\lambda)$, we have
	$$\int_M e^{-nu_\lambda}\ud\mu_g=\int_Me^{-nu_\lambda}\ud\mu_g\int_M e^{nu_\lambda}\ud\mu_g\leq e^{C_1\int_Mu_\lambda P_g u_\lambda \ud\mu_g+C}$$
	where $C_1$  and $C$ are both positive constants just depending  on $(M,g)$.
	The positivity of $P_g$ shows that
	$$\|u_\lambda-\bar u_\lambda\|^2_{L^2(M)}\leq C\int_M u_\lambda P_g u_\lambda\ud\mu_g.$$
	We could combine these estimates to get
	$$\alpha_\lambda\leq C\frac{\| f_0\|_{H^n}}{\|f_\lambda\|^2_{L^1}}\sqrt{\beta(\lambda)}e^{C_1\beta(\lambda)}.$$
	If $\lambda\in \left[a,b\right]\subset(0,-\bar f_0)$, we just need construct a function $u$ such $\int_M f_ae^{2u}\ud\mu_{g_b}>0$ to control the upper bound  $\beta_\lambda$ with help of the trick in Lemma \ref{trick lemma} since 
	$$\int_M f_\lambda e^{nu}\ud\mu_g\geq\int_M f_ae^{nu}\ud\mu_g.$$
	The construction of such function is easy and we omit the details. Thus $\alpha_\lambda$ and $\beta_\lambda$ are both uniformly bounded on compact subset of $(0,-\bar f_0)$.
	
	Consider  $\delta$ close to $0$ and $\mu$ close to $\lambda$ with $\mu,\lambda\in[a, b]$
	\begin{align*}
		\int_Mf_\mu e^{n(1+\delta)u_\lambda}\ud\mu_g=&\int_M(f_\lambda+(\mu-\lambda)) e^{n(1+\delta)u_\lambda}\ud\mu_g  \\
		=&\int_Mf_\lambda e^{nu_\lambda}(e^{n\delta u_\lambda}-1)\ud\mu_g\\
		&+(\mu-\lambda)\int_Me^{nu_\lambda}(e^{2\delta u_\lambda}-1)\ud\mu_g+\mu-\lambda\\
		=&n\delta\int_Mf_\lambda e^{nu_\lambda}u_\lambda\ud\mu_g+O(\delta^2)+(\mu-\lambda)O(\delta)+\mu-\lambda\\
		=&\frac{n\beta(\lambda)}{\alpha_\lambda}\delta+O(\delta^2)+(\mu-\lambda)O(\delta)+\mu-\lambda
	\end{align*}
	where we have used $u_\lambda\in \mathcal{M}^*(\lambda)$ and the equation \eqref{equ:u_lambda}. Hence we can choose $$\delta=-\frac{\alpha_\lambda}{n\beta(\lambda)}(\mu-\lambda)+O((\mu-\lambda)^2)$$ such that
	$$	\int_Mf_\mu e^{n(1+\delta)u_\lambda}\ud\mu_{g_b}=0.$$
	Then
	$$\beta(\mu)\leq(1+\delta)^2\beta(\lambda)=\beta(\lambda)-\frac{2\alpha_\lambda}{n}(\mu-\lambda)+O((\mu-\lambda)^2).$$
	Similarly, we have
	$$\beta(\lambda)\leq\beta(\mu)-\frac{2\alpha_\mu}{n}(\lambda-\mu)+O((\mu-\lambda)^2).$$ Hence 
	\begin{equation}\label{beta_lambda-beta_mu}
		-\frac{2\alpha_\mu}{n}(\mu-\lambda)+O((\mu-\lambda)^2)\leq\beta(\mu)-\beta(\lambda)\leq-\frac{2\alpha_\lambda}{n}(\mu-\lambda)+O((\mu-\lambda)^2).
	\end{equation}
	Due to the estimate \eqref{beta_lambda-beta_mu}, we show that $\beta(\lambda)$ is Lipschitz continuos on compact subset of $(0,-\bar f_0)$. Moreover, we have
	$$\frac{\ud\beta}{\ud\lambda}(\lambda_+)\leq -\frac{2\alpha_\lambda}{n},
	\quad
	\frac{\ud\beta}{\ud\lambda}(\lambda_-)\geq -\frac{2\alpha_\lambda}{n}.$$
	Since $\beta(\lambda)$ is differentiable almost everywhere, one has
	$$\beta'(\lambda)=-\frac{2\alpha_\lambda}{n}\quad a.e.$$
\end{proof}

\begin{lemma}\label{lem: beta(lambda) upper boud}
	There hold  $$ \lim_{\lambda\downarrow 0}\inf \frac{\beta(\lambda)}{\log(1/\lambda)}\geq \frac{(n-1)!|\mathbb{S}^n|}{n},\quad \lim_{\lambda\downarrow 0}\sup\frac{\beta(\lambda)}{\log(1/\lambda)}\leq (n-1)!|\mathbb{S}^n|.$$
	Moreover, if there exists a maximum point of $f_0(x)$ is flat up to $n-1$ order, one has
	$$\lim_{\lambda\downarrow 0}\sup\frac{\beta(\lambda)}{\log(1/\lambda)}\leq \frac{2}{n}(n-1)!|\mathbb{S}^n|,\quad \lim_{\lambda\downarrow 0}\inf \lambda\alpha_\lambda\leq (n-1)!|\mathbb{S}^n|.$$
	
\end{lemma}
\begin{proof}

	Firstly,  following the argument of Lemma 4.1 in \cite{NZ19}, H\"older's inequality together with  Adams-Fontana inequality(See Proposition 2.2 in \cite{Nid}) shows that
	\begin{align*}
		0<\|f_0\|^2_{L^1(M)}\leq&\int_M|f_0|e^{-nu_\lambda}\ud\mu_g\int_M |f_0|e^{nu_\lambda}\ud\mu_g\\
		=&\int_M|f_0|e^{-nu_\lambda}\ud\mu_g\int_M\lambda e^{nu_\lambda}\ud\mu_g\\
		\leq &C\lambda\int_Me^{-nu_\lambda}\ud\mu_g\int_Me^{nu_\lambda}\ud\mu_g\\
		=&C\lambda\int_Me^{n(-u_\lambda+\bar u_\lambda)}\ud\mu_g\int_Me^{n(u_\lambda-\bar u_\lambda)}\ud\mu_g\\
		\leq &C\lambda \exp\left(\frac{n}{(n-1)!|\mathbb{S}^n|}\beta(\lambda)\right).
	\end{align*}
	Then
	\begin{equation}\label{beta(lambda)+loglambda geq C}
		\beta(\lambda)-\frac{(n-1)!|\mathbb{S}^n|}{n}\log\frac{1}{\lambda}\geq C.
	\end{equation}
Immediately, there holds
$$\lim_{\lambda\downarrow 0}\inf \frac{\beta(\lambda)}{\log1/\lambda}\geq \frac{(n-1)!|\mathbb{S}^n|}{n}.$$
Of course, we also have
	\begin{equation}\label{beta_lambda tend to infty}
	\beta(\lambda)\to \infty,\quad \mathrm{as}\quad \lambda\to 0.
\end{equation}
	
	As for the upper bound for $\beta(\lambda)$, consider a test function introduced  in \cite{DM}(See Section 4 in \cite{DM})). For $\delta>0$ small to be chosen later, consider a non-decreasing smooth cut-off function $\chi_\delta(t)$ such that $\chi_\delta(t)=t$ for $t\in [0,\delta]$ and $\chi_\delta(t)=2\delta$ for any $t\geq 2\delta.$ For $s>0$ to be chosen later, define the test function $\varphi_{s,\delta}$ as
	\begin{equation}\label{def:varphi}
		\varphi_{s,\delta}(x)=\log \frac{2s}{1+s^2\chi_\delta(d_g(x,x_0))^2}
	\end{equation}
where $x_0$ is a maximum point of $f_0$ and $d_g(x,y)$ is the distance function on $(M,g)$.
 Letting  $s\to\infty$, Lemma 4.2 in \cite{DM}  and Lemma 4.5 in \cite{Nid}  show that
 \begin{equation}\label{energy of varphi_lambda}
 	\int_M\varphi_{s,\delta}P_g\varphi_{s,\delta}\ud\mu_g\leq (2(n-1)!|\mathbb{S}^n|+o_\delta(1))\log s+ C_\delta.
 \end{equation}
 Due to $f_0(x_0)=\max_M f_0=0$,  we could find $c_1>0$ and small $\delta>0$  such that
 $$f(x)\geq -c_1d_g(x,x_0)^2$$
 for any $x\in B_{2\delta}(x_0)$.  
Thus for such  small $\delta>0$ fixed
\begin{align*}
	&\int_Mf_\lambda e^{n\varphi_{s,\delta}}\ud\mu_g\\
	=&\int_{B_{2\delta}(x_0)}(f_0+\lambda)\left(\frac{2s}{1+s^2\chi_\delta(d_g(x,x_0))^2}\right)^n\ud\mu_g+\int_{M\backslash B_{2\delta}(x_0)}(f_0+\lambda))\left(\frac{2s}{1+4s^2\delta^2}\right)^n\ud\mu_g\\
	\geq &\int_{B_{\delta}(x_0)}\lambda\left(\frac{2s}{1+s^2d_g(x,x_0)^2}\right)^n\ud\mu_g-c_1\int_{B_{\delta}(x_0)}d_g(x,x_0)^2\left(\frac{2s}{1+s^2d_g(x,x_0)^2}\right)^n\ud\mu_g\\
	&-c_1\int_{B_{2\delta}(x_0)\backslash B_\delta(x_0)}d_g(x,x_0)^2\left(\frac{2s}{1+s^2\delta^2}\right)^n\ud\mu_g+O(s^{-n})\\
	\geq  & C_1\lambda\int_{|x|\leq \delta}\frac{s^n}{(a_1+s^2|x|^2)^n}\ud x-C_2\int_{|x|\leq \delta}\frac{s^n|x|^2}{(a_2+s^2|x|^2)^n}\ud x+O(s^{-n})
\end{align*}
where $C_1, C_2, a_1$ and $a_2$ are all  positive constants independent  of $s$ by considering the exponential map in the  small neighborhood.
Now  using polar coordinates, we have
\begin{align*}
	&\int_Mf_\lambda e^{n\varphi_{s,\delta}}\ud\mu_g\\
	\geq& C_3\lambda\int_0^{s\delta}\frac{t^{n-1}}{(a_1+t^2)^n}\ud t-C_4s^{-2}\int_0^{s\delta}\frac{t^{n+1}}{(a_2+t^2)^n}\ud t+ O(s^{-n})\\
=&C_3\lambda\int_0^{\infty}\frac{t^{n-1}}{(a_1+t^2)^n}\ud t-C_4s^{-2}\int_0^{\infty}\frac{t^{n+1}}{(a_2+t^2)^n}\ud t+ O(s^{-n}).
\end{align*}
We choose 
\begin{equation}\label{s choice}
	s^{-2}=c\lambda
\end{equation} where $c$ is a positive constant such that
$$\int_M(f_0+\lambda )e^{n\varphi_{s,\delta}}\ud\mu_g\geq Cs^{-2}+O(s^{-n}).$$
Hence if $\lambda$ small enough, we have
$$\int_M(f_0+\lambda )e^{n\varphi_{s,\delta}}\ud\mu_g>0.$$
Due to $\bar f_\lambda<0$, the same trick in Lemma \ref{trick lemma} implies that there exists $t_1\in (0,1)$ such that $t_1\varphi_{s,\delta}\in \mathcal{M}(\lambda)$.
Thus with help of \eqref{s choice} and \eqref{energy of varphi_lambda}, there holds 
$$\beta(\lambda)\leq E(t_1\varphi_{s,\delta})\leq ((n-1)!|\mathbb{S}^n|+o_{\delta}(1))\log\frac{1}{\lambda}+C(\delta).$$
Immediately, one has 
$$\lim_{\lambda\downarrow 0}\sup\frac{\beta(\lambda)}{\log1/\lambda}\leq (n-1)!|\mathbb{S}^n|+o_\delta(1).$$
Then by letting   $\delta\to 0$, we obtain  the desired result.

Now, if a maximum point $x_0$ of $f_0$ is flat up to $n-1$ order, we have
$$f(x)\geq -Cd_g(x,x_0)^n$$
in the neighborhood $B_{2\delta}(x_0).$ Similarly, a direct computation yields that 
\begin{equation}\label{n-2 flatness}
\int_M(f_0+\lambda)e^{n\varphi_{s,\delta}}\ud\mu_g\\
\geq C\lambda\int_0^{s\delta}\frac{t^{n-1}}{(a_1+t^2)^n}\ud t-Cs^{-n}\int_0^{s\delta}\frac{t^{2n-1}}{(a_2+t^2)^n}\ud t+ O(s^{-n}).
\end{equation}
Now for any small $\epsilon>0$, we choose
$$s^{-n+\epsilon}=\lambda$$
such that
$$\int_M(f_0+\lambda)e^{n\varphi_{s,\delta}}\ud\mu_g\geq Cs^{-n+\epsilon}+O(s^{-n}\log s).$$
Hence if $\lambda$ small enough, we have
$$\int_M(f_0+\lambda )e^{n\varphi_{s,\delta}}\ud\mu_g>0.$$
Samely as before, 
one concludes that 
$$
	\lim_{\lambda\downarrow 0}\sup\frac{\beta(\lambda)}{\log1/\lambda}\leq \frac{2}{n-\epsilon}(n-1)!|\mathbb{S}^n|+o_\delta(1).
$$
Letting  $\delta\to 0$ and then $\epsilon\to 0$, one has 
\begin{equation}\label{equ:sup beta/log lambda}
	\lim_{\lambda\downarrow 0}\sup\frac{\beta(\lambda)}{\log1/\lambda}\leq \frac{2}{n}(n-1)!|\mathbb{S}^n|.
\end{equation}
Meanwhile,  with help of Lemma \ref{Lemma beta_lambda derivative} and the above estimate  \eqref{equ:sup beta/log lambda}, we have
\begin{equation}\label{lambda alpha_lambda}
	\lim_{\lambda\downarrow 0}\inf \lambda\alpha_\lambda\leq \frac{n}{2}\lim_{\lambda\downarrow 0}\inf\lambda|\frac{\ud \beta_{\lambda}}{\ud \lambda}| \leq \frac{n}{2}\lim_{\lambda\downarrow 0}\sup\frac{\beta(\lambda)}{\log1/\lambda}
	\leq (n-1)!|\mathbb{S}^n|.
\end{equation}
Indeed, suppose $\lambda_1>0$ and $c_0>\lim_{\lambda\downarrow 0}\sup\frac{\beta(\lambda)}{\log1/\lambda}$ for almost $0<\lambda<\lambda_1$ such that $|\beta'_\lambda|\geq\frac{c_0}{\lambda}$ then for any sufficiently small $\lambda$ we obtain
$$\beta_\lambda-\beta_{\lambda_1}\geq\int^{\lambda_1}_\lambda|\beta_s'|\ud s\geq c_0\int^{\lambda_1}_\lambda\frac{\ud s}{s}\geq c_0\log(1/\lambda)+C$$
which contradicts to our choice $c_0$ by letting $\lambda\downarrow 0$. Finally, we complete our proof.
\end{proof}

\section{Blow-up analysis}\label{section 3}
The following lemma generalizes a  result  of Malchiodi(See Porposition 3.1 in \cite{Ma}).
\begin{lemma}\label{lem:h_k^+ geq (n-1)!/2 general case}
	Let $(M,g)$ be a compact closed manifold with even dimension $n\geq 4$ and  $\mathrm{ker} P_g=\{constants\}$. 
	Given a sequence  $\{u_k\}$ satisfying  $P_g u_k=h_k$ with $h_k\in L^1(M)$ and  $\int_Mh_k^+\ud\mu_g\leq C$ where $C$ is a constant independent of $k$.  Set $\bar u_k=\frac{1}{vol(M,g)}\int_M u_k\ud\mu_g$.
	Then one of the following is true: 
	\begin{enumerate}[(a)]
		\item 
		There exists a constant $q_0>n$ such that
		$$\int_Me^{q_0(u_k-\bar u_k)}\ud\mu_g\leq C.$$
		Moreover, if $	h^+_k\leq Ce^{nu_k}$, there holds
		$$\|(u_k-\bar u_k)^+\|_{L^\infty(M)}\leq C.$$
		\item There exist finite points $p_i$ with $1\leq i\leq i_0$ such that for any $s>0$
		\begin{equation}\label{B_s h^+}
			\lim_{k\to\infty}\inf\int_{B_s(p_i)}h^+_k\ud\mu_g\geq \frac{1}{2}(n-1)!|\mathbb{S}^n|
		\end{equation}
	Moreover, if $	h^+_k\leq Ce^{nu_k}$,  on any $B_r(x) \subset\subset M\backslash\{p_1,\cdots,  p_{i_0}\}$, there holds
	$$\|(u_k-\bar u_k)^+\|_{L^\infty(B_r(x))}\leq C.$$
	\end{enumerate}
\end{lemma}

\begin{proof}
	
	Due to $\mathrm{ker}P_g=\{constants\}$, there exists a  Green's function (See \cite{CY95}, \cite{Nid}, \cite{Ma}) satisfying
	$$P_gG(x,y)=\delta_y(x)-1$$
	with 
	$$|G(x,y)-\frac{2}{(n-1)!|\mathbb{S}^n|}\log\frac{1}{d_g(x,y)}|\leq C$$
	near the diagonal of $M\times M$.  Since $M$ is compact
	and $P_g(u+C)=P_gu$ for any constant $C$, we could choose a positive Green's function. From now on, we assume $G(x,y)>0.$ 
	Based on our assumptions, there holds
	\begin{equation}\label{h_k^+e^{nu_k}}
		\int_Mh_k^+\ud\mu_g\leq C.
	\end{equation}
	For  any $x_0\in M$, suppose that  there exists $r_0>0$ and $a_0>0$ such that
	\begin{equation}\label{equ: B_3r_0 h^+e2u}
		\int_{B_{3r_0}(x_0)}h^+_k\ud\mu_g\leq \frac{1}{2}(n-1)!|\mathbb{S}^n|-a_0
	\end{equation}
	for sufficiently large $k$. 
	Then,  for any $x\in B_{r_0}(x_0)$,
	we have
	\begin{align*}
		u_k(x)-\bar u_k
		=&\int_M G(x,y)h_k\ud\mu_g\\
		\leq &\int_M G(x,y)h^+_k\ud\mu_g\\
		=&\int_{B_{2r_0(x_0)}}G(x,y)h_k^+(y)\ud\mu_g(y)\\
		&+\int_{M\backslash B_{2r_0}(x_0)}G(x,y)h^+_k(y)\ud\mu_g(y)\\
		\leq &\int_{B_{2r_0(x_0)}}G(x,y)h_k^+\ud\mu_{g_b}(y)+C
	\end{align*}
	since $|G(x,y)|\leq C$ on $B_{r_0}(x_0)\times (M\backslash B_{2r_0}(x_0))$ and the estimate \eqref{h_k^+e^{nu_k}}. 
 Choose a cut-off function $0\leq \eta\leq 1$ with $\eta\equiv1 $ in $B_{2r_0}(x_0)$ and $\eta$ vanishes outside $B_{3r_0}(x_0)$.
	For $\alpha>0$, take the  strategy used  in \cite{BM}  and make use of  Jensen's inequality to get
	\begin{align*}
		&\int_{B_{r_0}(x_0)}e^{\alpha(u_k(x)-\bar u_k)}\ud\mu_g\\
		\leq &C\int_{B_{r_0}(x_0)}\exp\left(\int_{B_{2r_0}(x_0)}\alpha G(x,y)h^+_k(y)\ud\mu_{g}(y)\right)\ud\mu_g(x)\\
		\leq& C\int_{B_{r_0}(x_0)}\exp\left(\int_M\alpha G(x,y) h^+_k(y)\eta(y)\ud\mu_g(y)\right)\ud\mu_g(x)\\
		\leq & C\int_{B_{r_0}(x_0)}\int_M \exp\left(\alpha\|h^+_k\eta\|_{L^1(M)}G(x,y)\right)\frac{h^+_k\eta}{\|h_k^+\eta\|_{L^1(M)}}\ud\mu_{g}(y)\ud\mu_g(x)\\
		\leq& C\int_M\frac{h^+_k\eta}{\|h^+_k\eta\|_{L^1(M)}}\int_M\left(\frac{1}{d_g(x,y)}\right)^{\frac{2\alpha\|h^+_k\eta\|_{L^1(M)}}{(n-1)!|\mathbb{S}^n|}}\ud\mu_g(x)\ud\mu_g(y).
	\end{align*}
	With help of \eqref{equ: B_3r_0 h^+e2u},   there exists $\alpha_0>n$ such that the last integral is finite i.e.
	\begin{equation}\label{e^alpha u_k-bar u_k}
		\int_{B_{r_0}(x_0)}e^{\alpha_0(u_k-\bar u_k)}\ud\mu_g\leq C.
	\end{equation}
	If any $x\in M$, the estimate \eqref{equ: B_3r_0 h^+e2u} holds. Since $M$ is compact, there exist  finite balls covering  $M$ and then
	there exists $q_0>n$ such that
	$$\int_Me^{q_0(u_k-\bar u_k)}\ud\mu_g\leq C.$$
	
	Otherwise, due to $\int_M h_k^+\ud\mu_g\leq C$, there are finitely many points such \eqref{B_s h^+}  holds.
	
	Now, if the condtion $h_k^+\leq Ce^{nu_k}$ satisfies in addition,  for each $x_0$ as before such \eqref{equ: B_3r_0 h^+e2u} holds. 
	Then  for any $x\in B_{r_0/2}(x_0)$, with help of \eqref{e^alpha u_k-bar u_k}, one has 
	\begin{align*}
		u_k(x)-\bar u_k
		\leq &C+C\int_{B_{r_0}(x_0)}G(x,y) e^{nu_k}\ud\mu_g(y)\\
		\leq &C+C\left(\int_{B_{r_0}(x_0)}e^{\alpha_0 u_k}\ud\mu_{g}\right)^{n/\alpha_0}\left(\int_{B_{r_0}(x_0)}|G|^{\frac{1}{1-n/\alpha_0}}\ud\mu_{g}\right)^{1-n/\alpha_0}\\
		\leq & C
	\end{align*}
	which concludes that 
	\begin{equation}\label{u_k -bar u_k bound1}
		\|(u_k-\bar u_k)^+\|_{L^\infty(B_{r_0/2}(x_0))}\leq C.
	\end{equation}
	If for any $x\in M$, the estimate \eqref{equ: B_3r_0 h^+e2u} holds. With help of a finite covering, 
	one has
	\begin{equation}\label{u_k-bar u_k^+1}
		\|(u_k-\bar u_k)^+\|_{L^\infty(M)}\leq C.
	\end{equation}

Finally, we finish our proof.

\end{proof}
From now on, we assume $f_0$ has a $l$-type maximum point with $l\geq n$. Based on  the definition, it is not hard to see that  such $l$-type  maximum point is flat up to $n-1$ order. Meanwhile, we have 
$$|f_\lambda|\leq \lambda-f_0=2\lambda-f_\lambda,\quad f_\lambda^+\leq \lambda$$
and then the estimate \eqref{lambda alpha_lambda} yields that
\begin{equation}\label{|Q_k|}
	\lim_{\lambda\downarrow 0}\inf\int_M\alpha_\lambda|f_\lambda|e^{nu_\lambda}\ud\mu_g\leq 2(n-1)!|\mathbb{S}^n|,
\end{equation}
\begin{equation}\label{Q_k^+ intergral}
	\lim_{\lambda\downarrow 0}\inf\int_M\alpha_\lambda f_\lambda^+e^{nu_\lambda}\ud\mu_g\leq (n-1)!|\mathbb{S}^n|.
\end{equation}
Based on \eqref{lambda alpha_lambda},  we choose a subsequence $\lambda_k\to 0$ such that
\begin{equation}\label{lambda alpha upper bound}
	\lim_{k\to\infty}\lambda_k\alpha_{\lambda_k}\leq (n-1)!|\mathbb{S}^n|.
\end{equation}
For simplicity, set $u_k=u_{\lambda_k}$, $g_k=e^{2u_k}g$ and $Q_k:=\alpha_{\lambda_k}f_{\lambda_k}\leq \alpha_{\lambda_k}\lambda_k$. Recalling the equation \eqref{equ:u_lambda}, there holds 
$$P_gu_k=Q_ke^{nu_k}.$$

\begin{lemma}\label{lem:Q_k geq (n-1)!/2}
	Given $\{u_k\}$ as above, we have $$\bar u_k=\frac{1}{vol(M,g)}\int_Mu_k\ud\mu_g\to-\infty,\quad \alpha_{\lambda_k}\to \infty, \quad \mathrm{as}\; k\to\infty.$$ Moreover, there exist finite points $p_i$ with $1\leq i\leq i_0$ with $ i_0\in\{1,2\}$ and $f_0(p_i)=0$ such that for any $s>0$
	\begin{equation}\label{B_s Q^+}
		\lim_{k\to\infty}\inf\int_{B_s(p_i)}Q^+_ke^{nu_k}\ud\mu_g\geq \frac{1}{2}(n-1)!|\mathbb{S}^n|
	\end{equation}
as well as 
$$	\lim_{k\to\infty}\inf\int_{B_s(p_i)}e^{nu_k}\ud\mu_g\geq \frac{1}{2},\quad \lim_{k\to\infty}\inf\lambda_k\alpha_{\lambda_k}\geq\frac{1}{2}(n-1)!|\mathbb{S}^n|.$$
\end{lemma}

\begin{proof}
		Firstly, we claim that 
		\begin{equation}\label{bar u_k to-infty}
			\bar u_k\to-\infty,\quad \mathrm{as}\quad k\to\infty.
		\end{equation}
		Based on $u_\lambda \in \mathcal{M}^*(\lambda)$, there holds
		 $$\int_M (-f_0)e^{nu_k}\ud\mu_g=\lambda_k\int_Me^{nu_k}\ud\mu_g=\lambda_k$$
	and apply Jensen's inequality to get
	$$\exp\left(\int_Mnu_k\frac{(-f_0)}{\|f_0\|_{L^1(M)}}\ud\mu_g\right)\leq \int_Me^{nu_k}\frac{-f_0}{\|f_0\|_{L^1(M)}}\ud\mu_g=\frac{\lambda_k}{\|f_0\|_{L^1(M)}}.$$
	Hence we have
	$$\int_M(-f_0)u_k\ud\mu_g\to-\infty,\quad \mathrm{as}\quad k\to\infty.$$
	Notice that
	$$|\int_M(-f_0)u_k^-\ud\mu_g+\int_M(-f_0)u_k\ud\mu_g|=\int_M(-f_0)u_k^+\ud\mu_g\leq \frac{\max_M(-f_0)}{n}\int_Me^{nu_k}\ud\mu_g\leq C,$$
	and 
	$$\int_M(-f_0)u_k^-\ud\mu_g\leq \max_M(-f_0)\int_Mu_k^-\ud\mu_g$$ which yields 
	$$\int_Mu_k^-\ud\mu_g\to+\infty,\quad \mathrm{as}\quad k\to\infty.$$
	Then with help of the fact $x\leq e^x$,
\begin{align*}
		\int_Mu_k\ud\mu_g=&\int_Mu_k^+\ud\mu_g-\int_Mu_k^-\ud\mu_g\\
		\leq&\frac{1}{n}\int_Me^{nu_k}\ud\mu_g-\int_Mu_k^-\ud\mu_g\\
		=& \frac{1}{n}-\int_Mu_k^-\ud\mu_g\to-\infty.
\end{align*}
Thus we prove our claim \eqref{bar u_k to-infty}.

With help of \eqref{lambda alpha upper bound}, one has
$$\int_MQ_k^+e^{nu_k}\ud\mu_g\leq \lambda_k\alpha_{\lambda_k}\int_Me^{nu_k}\ud\mu_g\leq C, \quad Q_k^+e^{nu_k}\leq Ce^{nu_k}.$$
We will  apply Lemma \ref{lem:h_k^+ geq (n-1)!/2 general case} and rule out Case (a) in Lemma \ref{lem:h_k^+ geq (n-1)!/2 general case}. We argue by contradiction.  Supposing  that Case (a) in Lemma \ref{lem:h_k^+ geq (n-1)!/2 general case} holds, there holds
$$\|(u_k-\bar u_k)^+\|_{L^\infty(M)}\leq C$$
and then 
	\begin{equation}\label{bar u_k lower bound}
		1=\int_Me^{nu_k}\ud\mu_g=\int_Me^{n(u_k-\bar u_k)}\ud\mu_ge^{n\bar u_k}\leq Ce^{n\bar u_k}.
	\end{equation}
which contradicts to \eqref{bar u_k to-infty}.  Thus there are only finite  points $p_i\in M$  with $1\leq i\leq i_0$, $i_0\in \mathbb{N}$ such 
	\eqref{B_s Q^+} holds. 
	With help of the estimate
	$$\int_{B_s(p_i)}Q^+_ke^{nu_k}\ud\mu_g\leq\alpha_{\lambda_k}\lambda_k\int_Me^{nu_k}\ud\mu_g=\alpha_{\lambda_k}\lambda_k,$$
	 there holds
	$$\lim_{k\to\infty}\inf\alpha_{\lambda_k}\lambda_k\geq\frac{1}{2}(n-1)!|\mathbb{S}^n|.$$
	Immediately, 
	one has  
	$\alpha_{\lambda_k}\to\infty$ as $  \lambda_k\to 0.$
	Meanwhile,
	Due to $Q_k^+\leq \alpha_{\lambda_k}\lambda_k$, the estimates \eqref{lambda alpha upper bound} and \eqref{B_s Q^+} yield that
	\begin{equation}\label{e^nu_k geq 1/2}
		\lim_{k\to\infty}\inf \int_{B_s(p_i)}e^{nu_k}\ud\mu_g\geq \frac{1}{2}.
	\end{equation}
Based on the unit volume property
$\int_Me^{nu_k}\ud\mu_g=1,$
$i_0$ is either $1$ or $2$.

	The remaining task is to demonstrate   $f(p_i)=0$ for each $1\leq i\leq i_0$.
	We can establish this argument by means of a contradiction. Suppose that
	$f_0(p_i)<0$ for some $i$.
 Then, we arrive at a contradiction because  there exists $s>0$ such that for any $x\in B_s(p_i)$ and sufficiently  large $k$
	$$Q_k=\alpha_{\lambda_k}(\lambda_k+f_0(x))\leq 0,$$
	and then 
	$$\int_{B_s(p_i)}Q^+_ke^{nu_k}\ud\mu_g\leq 0$$
	which contradicts to \eqref{B_s Q^+}.
	
	In conclusion, we have successfully demonstrated all desired results. Thus, the proof is complete.
\end{proof}

\begin{remark}\label{Remark}
	Just as  Remark 3.2 in \cite{Ma}, using the same proof, we can  extend Lemma \ref{lem:Q_k geq (n-1)!/2} to the case in which also the metric on $M$ depends on $k$, and converges to some smooth $g$ in $C^m(M)$ for any integer $m$. We will use this variant later.
\end{remark}

Following the argument of Lemma 2.3 in \cite{Ma}, it is not hard to get the following lemma with help of \eqref{|Q_k|} and the representation of Green's function.
\begin{lemma}\label{lem: nabla ju}
	There is a constant depending only on $M$, $f_0$, $j$ and $p$ such that for any $1\leq j\leq n-1$ and $p>0$ satisfying $jp<n$,  small $r>0$ and any $x\in M$, there holds
	$$ \int_{B_r(x)}|\nabla ^ju_k|^p\ud\mu_g\leq Cr^{n-jp}.$$
\end{lemma}
\begin{proof}
	The proof is the same as Lemma 2.3 of \cite{Ma}. We omit the details here.
\end{proof}

With help of Lemma \ref{lem:Q_k geq (n-1)!/2}, there exists a  small $\delta>0$ such that $B_{2\delta}(p_i)$ are disjoint. For any blow-up point  $p_i$ (for simplicity we denote as $P$), we choose $r_k$ and $x_k$ such that
\begin{equation}\label{choice r_k and x_k}
	\int_{B_{r_k}(x_k)}e^{nu_k}\ud\mu_g=\sup_{x\in \bar{B}_{\delta}(P)}\int_{B_{r_k}(x)}e^{nu_k}\ud\mu_g=\frac{1}{8}.
\end{equation}
\begin{lemma}\label{lem:r_kto 0}
	There holds  
	$$r_k\to 0,\quad x_k\to P.$$
	Moreover, if $P$ is a $l$-type maximum point of $f_0$, one has  
	$$r_k^{l}\leq C\lambda_k,\quad d_g(x_k,P)^{l}\leq C\lambda_k. $$
\end{lemma}

\begin{proof}
	Firstly, based on  the choices of $x_i$ and $r_i$, we have
	$$\frac{1}{8}=\int_{B_{r_k}(x_k)}e^{nu_k}\ud\mu_g\geq \int_{B_{r_k}(P)}e^{nu_k}\ud\mu_g.$$ Then the estimate 
\eqref{e^nu_k geq 1/2} yields that $r_k\to 0$.
	With helpf of $r_k\to 0 $, if there is a subsequence $x_k\to x_0\in B_{\delta}(P)\backslash\{P\}$, for sufficiently large $k$, one has
	$$B_{r_k}(x_k)\subset\subset M\backslash\{p_i;1\leq i\leq i_0\}.$$ However,  making use of Lemma \ref{lem:Q_k geq (n-1)!/2}, there holds
	$$\int_{B_{r_k}(x_k)}e^{nu_k}\ud\mu_g\leq \int_{B_{r_k}(x_k)}e^{n(u_k-\bar u_k)}\ud\mu_ge^{n\bar u_k}\leq Ce^{n\bar u_k}\to 0$$
	as $k\to \infty$ which contradicts to our choice \eqref{choice r_k and x_k}.
	
	Suppose a subsequence $\frac{r^{l}_k}{\lambda_k}\to \infty$, based on the Definition \ref{def: 2l max point}, there exists
	a constant $C_1>0$ such that
	$$supp(Q^+_k)\cap B_{\delta}(P)\subset B_{C_1\lambda_k^{\frac{1}{l}}}(P).$$
 Then, with help of \eqref{lambda alpha upper bound} and \eqref{choice r_k and x_k}, for sufficiently large $k$,  there holds
\begin{align*}
	\frac{1}{4}(n-1)!|\mathbb{S}^n|\geq  &\alpha_{\lambda_k}\lambda_k\int_{B_{r_k}(x_k)}e^{nu_k}\ud\mu_g\\
	\geq &\alpha_{\lambda_k}\lambda_k\int_{B_{r_k}(P)}e^{nu_k}\ud\mu_g\\
	\geq &\int_{B_{r_k}(P)}Q^+_ke^{nu_k}\ud\mu_g\\
	\geq &\int_{B_{C_1\lambda_k^{1/l}}(P)}Q^+_ke^{nu_k}\ud\mu_g\\
	=&\int_{B_\delta(P)}Q_k^+e^{nu_k}\ud\mu_g
\end{align*}
which contradicts to Lemma \ref{lem:Q_k geq (n-1)!/2}.
Thus 
\begin{equation}\label{r_k upper bound}
	r_k^{l}\leq C\lambda_k.
\end{equation}
Due to $P$ is a $l$-type maximum point of $f_0$,  there exists $c_1>0$ such that for small $\delta>0$
$$-f_0(x)\geq c_1d_g(x,P)^{l},\quad x\in B_{\delta}(P).$$
Using the estimate \eqref{lambda alpha upper bound} and Lemma \ref{lem:Q_k geq (n-1)!/2}, there holds , for large $k$
$$C^{-1}\leq \lambda_k\alpha_{\lambda_k}\leq C.$$  With help of this fact, 
then there holds
\begin{align*}
	C\geq&\int_{B_{r_k}(x_k)}|Q_k|e^{nu_k}\ud\mu_g\\
	\geq &\int_{B_{r_k}(x_k)}\alpha_{\lambda_k}\left(-f_0-\lambda_k\right)e^{nu_k}\ud\mu_g\\
	\geq& c_1\int_{B_{r_k}(x_k)}\alpha_{\lambda_k}d_g(x,P)^{l}e^{2u_k}\ud\mu_g-C\\
	\geq &C\int_{B_{r_k}(x_k)}\lambda_k^{-1}d_g(x,P)^{l}e^{2u_k}\ud\mu_g-C
\end{align*}
which shows that
\begin{equation}\label{lambda_k^-1d_g(x,P)}
	\int_{B_{r_k}(x_k)}\lambda_k^{-1}d_g(x,P)^{l}e^{2u_k}\ud\mu_g\leq C.
\end{equation}
For $x\in B_{r_k}(x_k)$, with help of \eqref{r_k upper bound}, one has
\begin{equation}\label{lambda_k d_g(x,P)}
	\lambda_k^{-\frac{1}{l}}d_g(x,P)
	\geq\lambda_k^{-\frac{1}{l}}d_g(x_k,P)-\lambda_k^{-\frac{1}{l}}r_k
	\geq \lambda_k^{-\frac{1}{l}}d_g(x_k,P)-C.
\end{equation}
Combing these estimates  \eqref{choice r_k and x_k}, \eqref{lambda_k^-1d_g(x,P)} and \eqref{lambda_k d_g(x,P)}, we must have
$$\lambda_k^{-\frac{1}{l}}d_g(x_k,P)\leq C.$$
Finally, we complete our proof.
\end{proof}

\section{Proof of Main theorem}\label{section 4}
With help of higher order Bol's inequality in \cite{LW}, we are going to give the  proof  of Theorem \ref{main theorem}.

{\bf Proof of Theorem \ref{main theorem}:}
\begin{proof}
Following the strategy taken in \cite{Ma}, given small $\delta >0$, consider the exponential map
$$\exp_P:\hat B_{\delta}(0)\to M,\quad \exp_P(0)=P$$
where $\hat B_{\delta}(0):=\{z\in\mr^n||z|<\delta\}.$ We define
the metric on $\hat B_\delta(0)$ by $\tilde g_k:=(\exp_P)^*g$ and set
$$\tilde u_k=u_k\circ\exp_P,\quad z_k=\exp_P^{-1}(x_k)$$
where $x_k$ and $r_k$ come from \eqref{choice r_k and x_k}.
Consider the linear transformation $T_k:z\rightarrow r_kz+z_k$ and set
$$\hat u_k(z)=\tilde u_k(T_kz)+\log r_k$$
for $z\in D_k:=\{z\in \mr^n||z_k+r_kz|<\delta\}$.
Since $r_k\to 0$ and $z_k\to 0$(See Lemma \ref{lem:r_kto 0}), we find that $D_k$ exhaust $\mr^n$ as $k\to\infty$ and $\hat B_{\frac{\delta}{2r_k}}(0)\subset D_k$ for sufficiently large $k$.
Set $\hat g_k=r_k^{-2}T_k^*\tilde g_k$ and $\hat f_k(z)=\lambda_k+f_0(\exp_P(r_kz+z_k))$.
Due to the conformal property of $P_g$, there holds
$$P_{\hat g_k}\hat u_k(z)=\alpha_{\lambda_k}\hat f_ke^{n\hat u_k},\; \; z\in D_k.$$

Notice that $\hat g_k$ converges in $C^m_{loc}(\mr^n)$ to the flat metric $(dz)^2$ for any integer $m$.
By \eqref{choice r_k and x_k}, using a change of variables, we have
$$\frac{1}{8}=\int_{B_{r_k}(x_k)}e^{nu_k}\ud\mu_g=\int_{\hat B_1(0)}e^{n\hat u_k}\ud\mu_{\hat g_k}.$$ For any $y\in \hat B_{\frac{\delta}{2r_k}}(0)$, there holds
$$\int_{\hat B_{\frac{1}{2}}(y)}e^{n\hat u_k}\ud \mu_{\hat g_k}\leq \frac{1}{8}.$$
Then due to \eqref{lambda alpha upper bound}, 
for sufficiently large $k$, 
we have
\begin{equation}\label{int_Q_kleq 1/4}
	\int_{\hat B_{1/2}(y)}\hat Q_k^+e^{n\hat u_k}\ud\mu_{g_k}\leq \alpha_{\lambda_k}\lambda_k\int_{\hat B_{1/2}(y)}e^{n\hat u_k}\ud\mu_{g_k}\leq \frac{1}{4}(n-1)!|\mathbb{S}^n|
\end{equation}
where $\hat Q_k=\alpha_{\lambda_k}\hat f_k$.
Given $R>0$, define a smooth cut-off function $\Psi_R$ satisfying
\begin{align*}
	\left\{ 	\begin{array}{ll}
		\Psi_R(z)=1 &\mathrm{for}\; |z|\leq \frac{R}{2}\\
		\Psi_R(z)=0& \mathrm{for}\;|z|\geq R
	\end{array}	\right.
\end{align*}
We also set 
$$a_k=\frac{1}{|\hat B_R|_{\hat g_k}}\int_{\hat B_R}\hat u_k\ud \mu_{\hat g_k}$$
$$v_k=\Psi_R\hat u_k+(1-\Psi_R)a_k=a_k+\Psi_R(\hat u_k-a_k)$$
$$\hat v_k=v_k-a_k.$$
Notice that $\hat v_k$ is identically zero outside $\hat B_R$. By Lemma \ref{lem: nabla ju}, we have
\begin{equation}\label{nabla jv_k}
	\int_{\hat B_{2R}}\sum^{n-1}_{j=1}|\nabla^j\hat u_k|^p\ud \mu_{\hat g_k}\leq C_R,\quad p\in (1,\frac{n}{n-1}).
\end{equation} 
Due to $\hat v_k$ has a uniform compact support, the Poincar\'e inequality yields that
$$\int_{\hat B_R}|\hat v_k|^p\ud\mu_{\hat g_k}\leq C_R,\quad p\in (1,\frac{n}{n-1}).$$
Meanwhile,
$$P_{\hat g_k}\hat v_k=\Psi_R P_{\hat g_k}\hat u_k+L_k(\hat u_k-a_k)=\Psi_R\hat Q_ke^{nu_k}+L_k(\hat u_k-a_k)$$
where $L_k(\hat u_k-a_k)$ are linear opeartors which contains derivatives of order $0\leq j\leq n-1$ with uniformly bounded and smooth coefficients. Then
$$\int_{\hat B_{2R}}|L_k(\hat u_k-a_k)|^p\ud\mu_{\hat g_k}\leq C_R,\quad p\in (1,\frac{n}{n-1}).$$
Hence using Lemma \ref{lem:Q_k geq (n-1)!/2} and Remark \ref{Remark}, there exists $q>1$ such that
\begin{equation}\label{e^nqv_k}
	\int_{\hat B_R}e^{nq\hat v_k}\ud\mu_{\hat g_k}\leq C
\end{equation}
due to \eqref{int_Q_kleq 1/4}.
Actually, since $\hat v_k$ vanish identically outside $\hat B_R$, we can embed a fixed neighborhood  of $(\hat B_{2R}, \hat g_k) $ into a compact manifold, a torus for example, such its metric converges to a flat one.  

On the other hand,
$$a_k=\frac{1}{|\hat B_R|_{\hat g_k}}\int_{\hat B_R}\hat u_k\mu_{\hat g_k}\leq \frac{1}{n|\hat B_R|_{\hat g_k}}\int_{\hat B_R}e^{n\hat u_k}\mu_{\hat g_k}\leq C.$$
Since $v_k=\hat u_k$ in $\hat B_R$, there holds
$$\frac{1}{8}=\int_{\hat B_1(0)}e^{n\hat u_k}\mu_{\hat g_k}\leq e^{na_k}\int_{\hat B_R}e^{n\hat v_k}\ud\mu_{\hat g_k}\leq Ce^{n a_k}.$$
Thus $$|a_k|\leq C.$$ Combing with \eqref{e^nqv_k}, one has
$$\int_{\hat B_R}e^{nq\hat u_k}\ud\mu_{\hat g_k}\leq C$$
for some $q>1$. Thus standard elliptic theory yields that $\hat u_k$ is bounded in $W^{4,q}(\hat B_{R/2})$. By the arbitrary choice of $R$, up to a subsequence, Sobolev embedding theorem shows that $(\hat u_k)$ converge strongly  in $C^{\alpha}_{loc}(\mr^n)$ for some $\alpha\in (0,1)$ and also strongly in $H^{\frac{n}{2}}_{loc}(\mr^n)$ to a function $\hat u_\infty\in C^{\alpha}_{loc}(\mr^n)\cap H^{\frac{n}{2}}_{loc}(\mr^n)$. 

Due to Lemma \ref{lem:r_kto 0}, there are two possible cases. We will discuss them one by one.

Firslty, by the Definition \ref{def: 2l max point}, we have 
\begin{align*}
	\alpha_{\lambda_k}(\lambda_k+\hat f_0)=&\alpha_{\lambda_k}\lambda_k(1+\lambda_k^{-1}p_{l}(z_k+r_kz)+\lambda_k^{-1}O(|z_k+r_kz|^{l+1}))\\
	=&\alpha_{\lambda_k}\lambda_k(1+p_{l}(\lambda_k^{-\frac{1}{l}}z_k+\lambda_k^{-\frac{1}{l}}r_kz)+\lambda_k^{-1}|z_k+r_kz|^{l}O(|z_k+r_kz|))
\end{align*}
where $p_{l}(z)$ is a homogeneous function of degree $l$ satisfying $p_{l}(z)\leq 0$.
With help of Lemma \ref{lem:r_kto 0}, we have
$$\lambda_k^{-1}|z_k+r_kz|^{l}O(|z_k+r_kz|)\to 0$$
uniformly on $\hat B_R(0)$. Suppose, up to a subsequence, 
$\lambda_k\alpha_{\lambda_k}\to \mu$
with
\begin{equation}\label{mu leq (n-1)!S^n}
	\frac{1}{2}(n-1)!|\mathbb{S}^n|\leq \mu\leq(n-1)!|\mathbb{S}^n|
\end{equation}
due to Lemma \ref{lem:Q_k geq (n-1)!/2}.

Suppose  there is a subsequence such that 
$$\lim_{k\to\infty}\frac{r_k^{l}}{\lambda_k}\to r_0>0.$$
In addition, up to a subsequence, we have 
$$\lim_{k\to\infty}\lambda^{-\frac{1}{l}}_kz_k\to z_0\in \mr^n.$$ 
Thus
$$\lim_{k\to\infty}	\alpha_{\lambda_k}(\lambda_k+\hat f_0)=\mu(1+p_{l}(z_0+r_0z))$$
uniformly on $\hat B_R(0)$.
Then $\hat u_\infty\in H^{n/2}_{loc}(\mr^n)$ weakly solve the equation
\begin{equation}\label{equ: u_infty case2}
	(-\Delta)^{\frac{n}{2}}\hat u_\infty=\mu(1+p_{l}(z_0+r_0z))e^{n\hat u_\infty}
\end{equation}
satisfying
$$\int_{\mr^n}e^{n\hat u_\infty}\ud z\leq 1,\; \int_{\mr^n}|\mu(1+p_{l}(z_0+r_0z))|e^{n\hat u_\infty}\ud z<+\infty.$$
Meanwhile, with help of Lemma \ref{lem: nabla ju}, for any $R>0$, there holds 
$$\int_{B_R(0)}|\Delta\hat u_\infty|\ud z\leq CR^{n-2}$$
which concludes that
$$\frac{1}{|B_R(0)|}\int_{B_R(0)}|\Delta \hat u_\infty|\ud z\to 0,\quad \mathrm{as}\quad R\to \infty.$$
By a simple translation $w(z)=\hat u_\infty(z-\frac{1}{r_0}z_0)+\frac{1}{n}\log\mu$, it is not hard to check
$$(-\Delta)^{\frac{n}{2}}w(z)=(1+p_{l}(r_0z))e^{nw(z)}$$
satisfying 
$$\int_{\mr^n}e^{nw}\ud z\leq \mu,\; \int_{\mr^n}|(1+p_{l}(r_0z))|e^{nw}\ud z<+\infty$$
and 
$$\frac{1}{|B_R(0)|}\int_{B_R(0)}|\Delta w|\ud z\to 0,\quad \mathrm{as}\quad R\to \infty.$$ 
With help of Theorem 2.2 in \cite{Li 23 Q-curvature}, 
we show that $w$ is a normal solution.
Then applying higher order Bol's inequality Theorem 1.2 and Theorem 1.3 in \cite{LW}, 
there holds
$$\int_{\mr^n}e^{nw}\ud z>(n-1)!|\mathbb{S}^n|$$
which contradicts  to \eqref{mu leq (n-1)!S^n}.
Thus we rule out  this "slow bubble" case.

Hence we must have  $$\lim_{k\to\infty}\frac{r_k^{l}}{\lambda_k}=0.$$
Meanwhile,  on $\hat B_R(0)$, up to a subsequence,  there holds
\begin{equation}\label{mu_0 leq 0}
	p_{2l}(\lambda_k^{-\frac{1}{l}}z_k+\lambda_k^{-\frac{1}{l}}r_kz)\to \mu_0\leq 0.
\end{equation}
Thus $\hat u_\infty$ weakly solves the following equation
$$(-\Delta)^{\frac{n}{2}}\hat u_\infty=\mu(1+\mu_0)e^{n\hat u_\infty}\quad \mathrm{on}\;\;\mr^n$$
with the volume 
\begin{equation}\label{volume upper bound}
	\int_{\mr^n}e^{n\hat u_\infty}\ud z\leq 1.
\end{equation}
Samely as before, Lemma \ref{lem: nabla ju} shows that  
$$\int_{ B_{R}(0)}|\Delta \hat u_\infty|\ud z=o(R^n).$$ 
Theorem 2.2 in \cite{Li 23 Q-curvature} shows that $\hat u_\infty$  is a normal solution.
With help of Lemma 2.18 in \cite{Li 23 Q-curvature},   one has
$\mu_0>-1$. Due to the classification theorem in \cite{Lin}, \cite{WX}, and \cite{Mar MZ}, one has 
$$\mu(\mu_0+1)\int_{\mr^n}e^{n\hat u_\infty}\ud z=(n-1)!|\mathbb{S}^n|.$$
With help of  \eqref{mu leq (n-1)!S^n}, \eqref{volume upper bound} and \eqref{mu_0 leq 0},  we have 
$$\mu_0=0,\quad \mu=(n-1)!|\mathbb{S}^n|,\quad \int_{\mr^n}e^{n\hat u_\infty}\ud z=1.$$ In addition,
$$\hat u_\infty-\frac{1}{n}\log\frac{1}{|\mathbb{S}^n|}=\log\frac{2s}{s^2+|z-z_0|^2}$$
for some $s>0 $ and $z_0\in \mathbb{R}^n$.
Moreover, there holds
$$\lim_{k\to\infty}\lambda_k\alpha_{\lambda_k}=(n-1)!|\mathbb{S}^n|.$$
Due to \eqref{lambda alpha_lambda} and Lemma \ref{lem: beta(lambda) upper boud}, up to a subsequence,  there holds
$$\frac{n}{2}\lim_{k\to\infty}\frac{\beta(\lambda_k)}{\log1/\lambda_k}=(n-1)!|\mathbb{S}^n|.$$

Finally, we finish our proof.

\end{proof}

\vspace{3em}
{\bf Acknowledgements:}  The author would like to thank Professor Xingwang Xu, Professor Yuxin Ge, Professor Juncheng Wei  for helpful discussions.

\end{document}